\documentclass[leqno, 12pt]{amsart}

\usepackage{amsmath, amssymb, amsthm, cite, graphicx, float, mathrsfs, commath, dsfont, bbm, bm, mathtools, thmtools}
\usepackage{algorithm, algpseudocode}
\usepackage{tikz, enumitem}
\usetikzlibrary{patterns,snakes}
\usepackage{color}
\usepackage{optidef}
\usepackage[mathscr]{eucal}
\usepackage[sc]{mathpazo}

\renewcommand{\bar}{\overline}
\renewcommand{\hat}{\widehat}

\newcommand{\bB}{\mathbb{B}}

\newcommand{\bR}{\mathbb{R}}
\newcommand{\bN}{\mathbb{N}}

\newcommand{\R}{\bR}
\newcommand{\Rn}{\bR^n}

\newcommand{\eR}{{\overline\bR}}

\newcommand{\cP}{\mathcal{P}}

\newcommand{\cL}{\mathcal{L}}
\newcommand{\cC}{\mathcal{C}}

\newcommand{\cI}{\mathcal{I}}

\newcommand{\half}{\frac{1}{2}}

\DeclareMathOperator{\epi}{epi}
\DeclareMathOperator*{\argmin}{arg\,min}

\newcommand{\ip}[2]{\left\langle #1,\, #2\right\rangle}

\newcommand{\dom}[1]{\mathrm{dom}\left(#1\right)}

\newcommand{\bset}[2]{\left\{#1\,\left|\, #2\right.\right\}}

\newcommand{\indicator}[2]{\delta\left(#1\,\left|\, #2\right.\right)}

\DeclareMathOperator{\lev}{lev}

\newcommand{\ri}[1]{\mathrm{ri}\left(#1\right)}

\DeclareMathOperator{\aff}{aff}

\newcommand{\sd}{\partial}

\newcommand{\Del}{{\Delta}}

\newcommand{\eps}{\epsilon}

\newcommand{\bx}{\bar x}
\newcommand{\bt}{\bar t}

\newcommand{\bdel}{{\bar\delta}}

\newcommand{\beps}{{\bar\epsilon}}

\newcommand{\tM}{\widetilde M}

\newcommand{\hdel}{\hat\delta}

\newcommand{\Dcd}[1]{{D^{\scriptscriptstyle{C}}_{#1}}}
\newcommand{\Dco}{{D^{\scriptscriptstyle{C}}_{1}}}
\newcommand{\Deltac}[1]{{\Delta^{\scriptscriptstyle{C}}_{#1}}}
\newcommand{\Deltaco}{{\Delta^{\scriptscriptstyle{C}}_{1}}}

\usepackage[margin=1in, top=.8in, left=.8in]{geometry}
\usepackage{xcolor}
\usepackage{hyperref}
\hypersetup{ 
	colorlinks,
	linkcolor={red!50!black},
	citecolor={blue!50!black},
	urlcolor={blue!80!black}
}
\usepackage{cleveref}

\newtheorem{theorem}{Theorem}[section]
\newtheorem{corollary}{Corollary}[section]
\newtheorem{lemma}{Lemma}[section]
\newtheorem{definition}{Definition}[section]
\theoremstyle{definition}

\newtheorem{remark}{Remark}

\crefname{algocf}{alg.}{algs.}
\Crefname{algocf}{Algorithm}{Algorithms}

\algrenewcommand\algorithmicrequire{\textbf{Input:}}
\algrenewcommand\algorithmicensure{\textbf{Output:}}

\algnewcommand\algorithmicinput{\textbf{Initialize:}}
\algnewcommand\Initialize{\item[\algorithmicinput]}

\algnewcommand\algorithmicstepk{\textbf{Step }k\textbf{:}}
\algnewcommand\Step{\item[\algorithmicstepk]}

\usepackage[parfill]{parskip} 

\begin{document}
\title{Line Search and Trust-Region Methods for Convex-Composite Optimization}
\author{J.~V.~Burke}
\author{A.~Engle}
\thanks{University of Washington Department of Mathematics, Seattle, WA, \texttt{jvburke01@gmail.com}, \texttt{aengle2@uw.edu}}

\begin{abstract}
We consider descent methods for solving non-finite valued 
nonsmooth convex-composite optimization problems 
that employ Gauss-Newton subproblems to determine the iteration update. 
Specifically, we establish the global convergence properties for descent methods that
use a backtracking line search, a weak Wolfe line search, or a trust-region update.
All of these approaches are designed to 
exploit the structure associated with convex-composite problems.
\end{abstract}

\maketitle


\section{Introduction}

We consider three descent methods for solving the convex-composite optimization problem
\begin{mini}
	{x\in\R^n}{f(x):=h(c(x))+g(x),}{\tag{$\cP$}\label{theprogram}}{}
\end{mini}
where $h:\R^m\to\R$ is convex, $g:\R^n\to\eR$ is closed, proper, and convex, and $c:\R^n\to\R^m$ is $\cC^1$-smooth.
Our focus is on methods that employ search directions or steps  $d^k\in\R^n$  that approximate solutions to Gauss-Newton subproblems  
\begin{mini}
	{d\in\R^n}{\Delta f(x^k;d):=h(c(x^k)+\nabla c(x^k)d)+g(x^k+d) - f(x^k)}{\tag{$\cP_k$}\label{subprob}}{}
	\addConstraint{\norm{d}\leq\eta_k,}{}
\end{mini}
where $\set{x^k}\subset\dom{g}$ are the iterates generated by the algorithm, $\set{\eta_k}\subset(0,\infty]$, and $\Delta f(x;d)$ is an approximation to the directional derivative 
$f'(x;d)$ as in \cite{burke1985descent} 
(also see Lemma \ref{lem:deltaandf'}). 
By \emph{descent}, we mean that $d^k$ satisfies $\Delta f(x^k;d^k)<0$ at each iteration $k$. Two of the approaches are line search methods based on backtracking 
and weak Wolfe conditions. The third approach is a trust-region method. Although the weak Wolfe line search discussed here is motivated by this method for differentiable functions, as in \cite{lewis2013nonsmooth}, it is modified in a way that allows its application to nondifferentiable functions.

Algorithms for the
problem \ref{theprogram} have recently received renewed interest due to numerous modern applications in machine learning and nonlinear dynamics \cite{2017arXiv170502356D,aravkin2013sparse,davis2017nonsmooth, drusvyatskiy2018error,2016arXiv160500125D}. In a companion work, we establish 
conditions for the local super-linear and quadratic convergence of methods based
on \ref{subprob} when both $h$ and $g$ are assumed to be piecewise
linear-quadratic but not necessarily finite valued \cite{burke2018strong}.

Previously, the backtracking line search was studied in finite-valued case and in the 
absence of the function $g$ \cite{burke1985descent}. 
In recent work, Lewis and Wright \cite{lewis2016proximal} utilized a similar backtracking line search in the context of infinite-valued prox-regular composite optimization. 
Lewis and Overton \cite{lewis2013nonsmooth} developed a weak Wolfe algorithm using directional derivatives for finite-valued nonsmooth functions $f$ that are absolutely continuous along the line segment of interest, with finite termination in particular when the function $f$ is semi-algebraic. 
The method of Lewis and Overton can be applied in the finite-valued convex-composite case where $g=0$. 
Here, we develop a weak Wolfe algorithm for infinite-valued problems that uses the approximation $\Delta f(x;d)$ to the directional derivative, which exploits the structure associated with convex-composite problems.

The function $g$ in \ref{theprogram} is typically nonsmooth and is used to induce structure in the solution $\bx$. 
 For example, it can be used to introduce sparsity or group sparsity in the solution $\bx$ as well as bound constraints $\bx$.
Drusvyatskiy and Lewis \cite{drusvyatskiy2018error} have established local and global convergence of proximal-based methods for solving \ref{theprogram}, and Drusvyatskiy and Paquette \cite{2016arXiv160500125D} have established iteration complexity results for proximal methods to locate first-order stationary points for \ref{theprogram}. 

While the assumptions we use are similar to those in \cite{drusvyatskiy2018error,2016arXiv160500125D}, our algorithmic approach differs 
significantly. In particular, we use either a backtracking or an adaptive weak Wolfe line search, or a trust-region strategy to 
induce objective function descent at each iteration. 
In addition, we do not exclusively use proximal methods to generate search directions or employ the backtracking line search to estimate Lipschitz constants as in \cite{lewis2016proximal,2016arXiv160500125D}. 
Moreover, all of the methods discussed here make explicit use of the structure in \ref{theprogram}, thereby differing from the method developed in \cite{lewis2013nonsmooth}.

\section{Notation}
This section records notation and tools from convex and variational analysis used throughout the paper. Unless otherwise stated, we follow the notation of \cite{rockafellar_wets_1998,wright1999numerical,rockafellar2015convex, burke2018strong}.
\\
For any two points $x,x'\in\R^n$, denote the line segment connecting $x$ and $x'$ by $[x,x']:=\bset{(1-\lambda)x + \lambda x'}{0\leq\lambda\leq1}$. For a nonempty closed convex set $C\subset\R^m$ let $\aff{C}$ denote its \emph{affine hull}. Then the \emph{relative interior} of $C$ is
\[
\ri{C} = \bset{x \in \aff{C}}{\exists \,(\epsilon>0)\ (x+\epsilon\bB)\cap\aff{C}\subset C}.
\]
The functions in this paper take values in the extended reals $\eR:=\R\cup\set{\pm\infty}$. For $f:\R^n\to\eR$, the \emph{domain} of $f$ is $	\dom{f} := \bset{x\in\R^n}{f(x)<\infty}$,
and the \emph{epigraph} of $f$ is $\epi{f} := \bset{(x,\alpha)\in \R^n\times\R}{f(x)\le\alpha}$. 
\\
A function $f$ is \emph{closed} if the \emph{level sets} $\lev_f(\alpha):=\bset{x\in\R^n}{f(x)\le\alpha}$ are closed for all $\alpha\in\R$, \emph{proper} if $\dom{f}\neq\emptyset$ and $f(x)>-\infty$ for all $x\in\R^n$, and \emph{convex} if $\epi{f}$ is a convex subset of $\R^{n+1}$. For a set $X\subset \dom{f}$ and $\bx\in X$, the function $f$ is \emph{strictly continuous at $\bx$ relative to $X$} if
\begin{equation*}
\limsup_{\substack{x,x'\xrightarrow[X]{}\bar{x}\\x\neq x'}} \frac{\norm{f(x)-f(x')}}{\norm{x-x'}}<\infty,
\end{equation*}
where $x,x\xrightarrow[X]{}\bx\Longleftrightarrow x,x' \in X$ and $x,x'\to \bx$ represents \emph{converegence within $X$}.
This finiteness property is equivalent to $f$ being locally Lipschitz at $\bx$ relative to $X$ (see \cite[Section 9.A]{rockafellar_wets_1998}). By \cite[Theorem 10.4]{rockafellar2015convex}, proper and convex functions $g:\R^n\to\eR$ are strictly continuous relative to $\ri{\dom{g}}$.
To each nonempty closed convex set $C$, we associate the closed, proper, and convex \emph{indicator function} defined by
\begin{center}
	\(	\indicator{x}{C}:=\begin{cases}
	0 & x\in C,\\
	+\infty & x\not\in C.\end{cases}
	\)
\end{center}
Suppose $f:\R^n\to\eR$ is finite at $\bx$ and $w\in\R^n$. The \emph{subderivative} $\dif f(\bar{x}):\R^n\to\eR$ and \emph{one-sided directional derivative} $f'(\bx;\cdot)$ at $\bx$ for $w$ are
\begin{align*}
\dif f(\bar{x})(w) &:= \liminf_{\substack{t\searrow0\\ w'\to w}} \frac{f(\bar{x}+t w)-f(\bar{x})}{t}, &
f'(\bx;w) &:= \lim_{t\searrow0}\frac{f(\bx + t w) - f(\bx)}{t}.
\end{align*}
The structure of \ref{theprogram} allows the classical one-sided directional derivative $f'(\bx;\cdot)$ to capture the variational properties of its more general counterpart as discussed in the next section.
\\
Results in the following section also require the notion of \emph{set convergence} from variational analysis, as in \cite[Section 4.A]{rockafellar_wets_1998}. For a sequence of sets $\set{C_n}_{n\in\bN}$, with $C_n\subset\R^m$, the \emph{outer and inner limits} are defined, respectively, as
	\begin{align*}
		\limsup_{n\to\infty} C_n&:=\bset{x}{\exists\, (\text{infinite }K\subset\bN,\ x^k\xrightarrow[K]{} x)\ \forall\,(k\in K)\ x^k\in C_k}\\
		\liminf_{n\to\infty} C_n&:=\bset{x}{\exists\,(n_0\in\bN,\ x^n\to x)\ \forall\,(n\ge n_0)\ x^n\in C_n}.
	\end{align*}
	The sets $C_n$ \emph{converge} to a set $C$ if the two limits agree and equal $C$: \[
		\limsup_{n\to\infty} C_n = \liminf_{n\to\infty} C_n=C.
	\] With this notion of convergence in mind, we apply it to the epigraphs of a sequence of functions and say that $f^k:\R^m\to\eR$ \emph{epigraphically converge} to $f:\R^m\to\eR$, written $f^k\xrightarrow[]{e} f$, if and only if $\epi{f^k}\to\epi{f}$ (see \cite[Section 7.B]{rockafellar_wets_1998}).
\section{Properties of Convex-Composite Objectives}

The general convex-composite optimization problem \cite{burke1987second} is of the form
\begin{mini}
	{x\in\R^n}{f(x):=\psi(\Phi(x)),}{\label{eq:generalcvxcomp}}{}	
\end{mini}
where $\psi:\R^m\to\eR$ is closed, proper and convex, and $\Phi:\R^n\to\R^m$ is sufficiently smooth. Allowing infinite-valued convex functions $\psi$ into the composition introduces theoretical difficulties as discussed in \cite{rockafellar_wets_1998, lewis2016proximal,burke2018strong,lewis2002active,burke1992optimality}. In this work, we assume $f$ takes the form given in \ref{theprogram} by setting $\psi(y,x) := h(y) + g(x)$ and $\Phi(x) = (c(x),x)$. In this case, the calculus simplifies dramatically. As in \cite{burke1987second}, we have $\dom{f}=\dom{g}$ and
\begin{equation}\label{eq:littleoexpansion}
	f(x+d) = h(c(x)+\nabla c(x)d) + g(x+d) + o(\norm{d}).
\end{equation}
Consequently, at any $x\in\dom{g}$ and $d\in\R^n,\ f$ is directionally differentiable, with
\[
	\dif f(x)(d) = f'(x;d) = h'(c(x);\nabla c(x)d)+g'(x;d).
\]
This motivates defining the \emph{subdifferential} of $f$ at any $x\in\dom{g}$ by setting
\begin{equation}\label{eq:cvxcompsd}
	\sd{f}(x) := \nabla c(x)^\top\sd h(c(x)) + \sd{g}(x).
\end{equation}

Within the context of variational analysis \cite{rockafellar_wets_1998}, we have that $f$ is \emph{subdifferentially regular} on its domain and the subdifferential of $f$ as defined above agrees with the 
\emph{regular and limiting subdifferentials} of variational analysis. In particular, $f'(x;d) = \sup_{v\in \sd f(x)}\ip{v}{d}$.

Following \cite{burke1985descent}, we define an approximation to the directional derivative
that is key to our algorithmic development.
\begin{definition}
	Let $f$ be as in \ref{theprogram} and $x\in\dom{g}$. Define $\Delta f(x;\cdot):\R^n\to\eR$ by
	\begin{equation}
	\label{delta f}
		\Delta f(x;d) = h(c(x)+\nabla c(x)d) + g(x+d) - h(c(x)) - g(x).
	\end{equation}
\end{definition}
The next lemma records the interplay between $\Delta f(x;d)$ and its infinitesimal counterpart $f'(x;d)$ and is a consequence of \eqref{eq:littleoexpansion} and the definitions.
\begin{lemma}\label{lem:deltaandf'}
	Let $f$ be given as in \ref{theprogram} and let $x\in \dom{g}$. Then
	\begin{enumerate}[label=(\alph*)]
		\item the function $d\mapsto \Delta f(x;d)$ is convex;
		\item for any $d\in\R^n$, the difference quotients $\frac{\Delta f(x;td)}{t}$ are nondecreasing in $t>0$, with
		\begin{align*}
				f'(x;d) 
				&= \inf_{t>0} \frac{\Delta f(x; td)}{t},
		\end{align*}
		and in particular;
		\item for any $d\in \R^n,\ f'(x;d)\leq \Delta f(x;d)$;
		\item for any $d\in \R^n,\ t\in[0,1],\ \Delta f(x; td) \leq t\Delta f(x;d)$.
	\end{enumerate}
\end{lemma}
We now state equivalent first-order necessary conditions for a local minimizer $\bx$ of \ref{theprogram}, emphasizing that $f'(x;d)$ and $\Delta f(x;d)$ are interchangeable with respect to these conditions. The proof of this result parallels that given in \cite{burke1987second} using \eqref{eq:littleoexpansion} and \eqref{eq:cvxcompsd}.
\begin{theorem}[First-order necessary conditions for \ref{theprogram}] \cite[Theorem 2.6]{burke1987second}\label{thm:fonc}
Let $h$, $c$, and $g$ be as given in \ref{theprogram}. If $\bx\in\dom{g}$ is a local minimizer of \ref{theprogram}, then 
	\(
		f'(\bx;d)\ge0,\text{ for all }d\in\R^n.
	\)
	Moreover, the following conditions are equivalent for any $x\in\dom{g}$,
	\begin{enumerate}[label=(\alph*)]
		\item $0\in\sd f(x)$;
		\item for all $d\in\R^n,\ 0\leq f'(x;d)$;
		\item for all $d\in\R^n,\ 0\leq \Delta f(x;d)$;
		\item for all $\eta>0,\ d=0$ solves
		\(\min\{\Delta f(x;d)\, |\, \norm{d}\leq\eta\}\).
	\end{enumerate}
\end{theorem}
%
%
%
The next lemma shows that if 
the sequence $\{(x^k,d^k)\}\subset\R^n\times\R^n$ is such that
$d^k$ is an approximate solution to \ref{subprob} 
for all $k$ with $\Delta f(x^k;d^k)\to0$, then cluster points of $\set{x^k}$ are first-order stationary for \ref{theprogram}.
\begin{lemma}\label{lem:delfto0}
Let $h$, $c$, and $g$ be as in \ref{theprogram} and $\alpha\in\R$. Set $\cL:=\lev_f(\alpha)$. Let $\set{(x^k,\eta_k)}\subset\cL\times\R_+$, with $(x^k,\eta_k)\rightarrow (\bx,\bar\eta)\in\R^n\times\R_+$
and $0<\bar\eta<\infty$. 
Define
\begin{equation}\label{eq:epi functions}
\begin{aligned}
\Delta_k f(d)&:=\Delta f(x^k;d)+\delta_{\eta_k\bB}(d),\mbox{ and}
\\
\bar\Delta_k f&:=\min_d \Delta_k f(d)
\end{aligned}
\end{equation}

If, for each $k\geq1,\ d^k\in\eta_k\bB$ satisfies
\begin{equation}\label{eq:sandwich}
	\Delta f(x^k;d^k)\le\beta\bar\Delta_k f\leq0,
\end{equation}
with $\Delta f(x^k;d^k)\to0$, then $0\in \sd f(\bx)$.
\end{lemma}
\begin{proof}Since $f$ is closed, $f(\bx)\leq \alpha$, which implies $\bx\in\dom{g}$. Define the functions
	\begin{align*}
		h_k(d) &:= h(c(x^k)+\nabla c(x^k)d) - h(c(x^k)),\\
		h_\infty(d) &:= h(c(\bx)+\nabla c(\bx)d) - h(c(\bx)),\\
		g_k(d) &:= g(x^k+d)-g(x^k), \text{ and}\\ 
		g_\infty(d) &:= g(\bx + d)-g(\bx). 
	\end{align*}
Since $0<\eta_k\to\bar\eta$, with $\bar\eta>0$, and since $\indicator{d}{\eta_k\bB} = \indicator{\frac{1}{\eta_k}d}{\bB}$, \cite[Proposition 7.2]{rockafellar_wets_1998} implies
\[
	\indicator{\cdot}{\eta_k\bB} \xrightarrow[]{e}\indicator{\cdot}{\bar\eta\bB}.
\] By \cite[Exercise 7.8(d)]{rockafellar_wets_1998},	\(g_k\xrightarrow[]{e}g_\infty,\)
so \cite[Exercise 7.47]{rockafellar_wets_1998} implies $g_k + \indicator{\cdot}{\eta_k\bB}\xrightarrow[k\to\infty]{e}g_\infty + \indicator{\cdot}{\bar\eta\bB}$, and applying \cite[Exercise 7.47]{rockafellar_wets_1998} again yields
\[
	h_k + g_k + \indicator{\cdot}{\eta_k\bB}\xrightarrow[]{e} h_\infty + g_\infty + \indicator{\cdot}{\bar\eta\bB}.
\]
Equivalently,
\[
	\Delta f(x^k;\cdot) + \indicator{\cdot}{\eta_k\bB} \xrightarrow[]{e}\Delta f(\bx;\cdot)+\indicator{\cdot}{\bar\eta\bB}.
\]
By \cite[Proposition 7.30]{rockafellar_wets_1998} and \eqref{eq:sandwich}, 
\[
	0=\limsup_k \bar\Delta_kf\leq\min_{\norm{d}\leq\bar\eta}\Delta f(\bx;d)\leq 0,
\]
so \Cref{thm:fonc} implies $0\in\sd f(\bx)$.
\end{proof}
The approximate solution condition \eqref{eq:sandwich}
is described in \cite{burke1985descent}. It can be satisfied by employing 
the
trick described in \cite[Remark 6, page 343]{burke1989infeasible}.
Specifically, any solution
technique solving the convex subproblems \ref{subprob}
that also generates lower bounds $\ell_{k,j}\in\R$ such that $\ell_{k,j}\nearrow \bar\Delta_kf$ and 
$\Delta f(x^k;d^{k,j})\searrow\bar\Delta_kf$ as $j\to\infty$. 
If $\bar\Delta_k f<0$, then the condition
\[ 
	\Delta f(x^k;d^{k,j})\leq\beta\ell_{k,j}
\]
is finitely satisfied, and
\[
	\Delta f(x^k;d^{k,j})\leq\beta\bar\Delta_k f.
\]
We conclude this section with a mean-value theorem for \ref{theprogram}.
\begin{theorem}[Mean-Value for convex-composite]\cite[Theorem 10.48]{rockafellar_wets_1998}\label{lem:mvt}
	Let $f$ be as in \ref{theprogram}, $g$ strictly continuous relative to its domain, and $x_0,x_1\in \dom{g}$. Then there exists $t\in(0,1),\ x_t:=(1-t)x_0+tx_1$ and $v\in \sd f(x_t)$ such that
	\[
	f(x_1) - f(x_0) = \ip{v}{x_1-x_0}.
	\]
\end{theorem}
\begin{proof}
	Let $F(t):=(1-t)x_0 + tx_1$ and let $\varphi(t) = f(F(t))-(1-t)f(x_0)-tf(x_1)$. Then 
	\[
	\varphi(t) = h(c(F(t))) + g(F(t)) - (1-t)f(x_0)-tf(x_1)
	\]
	is an instance of \ref{theprogram}, since $g\circ F$ is convex. Consequently, the chain rules for $\varphi$ and $-\varphi$ on $[0,1]$ are
	\begin{align*}
	\sd \varphi(t) &= F'(t)^\top \nabla c(F(t))^\top \sd h(c(F(t))) + F'(t)^\top\sd g(F(t)) + f(x_0)-f(x_1)\\
	&= \bset{\ip{v}{x_1-x_0}}{v\in \sd f(F(t))}+f(x_0)-f(x_1),\text{ and}\\
	\sd(-\varphi)(t) &= F'(t)^\top \nabla c(F(t))^\top \sd (-h)(c(F(t))) + F'(t)^\top\sd (-g)(F(t)) + f(x_1)-f(x_0)\\
	&= \bset{\ip{-v}{x_1-x_0}}{v\in \sd f(F(t))}+f(x_1)-f(x_0).
	\end{align*}
	As $g$ is continuous on its domain, $\varphi$ is continuous on $[0,1]$ with $\varphi(0)=\varphi(1)=0$. Therefore,
	$\varphi$ attains either its minimum or maximum value at some $\bt\in(0,1)$, and $0\in\sd\varphi(\bt)$ or $0\in \sd (-\varphi)(\bt)$ respectively.
\end{proof}
\section{Backtracking for Convex-Composite Minimization}
The simplest and most well established line search is the Armijo-Goldstein backtracking procedure \cite{wright1999numerical}.
It has been adapted for the convex-composite setting in \cite{burke1985descent, powell1984global, yuan1983global} where it takes the form
\[
f(x+td) \le f(x) + \sigma_1 t\Delta f(x;d)
\]
and enforces \emph{sufficient decrease} of $f$ along the ray $\bset{x+td}{t>0}$, with $\Delta f$ acting as a surrogate for the directional derivative. Existence of step sizes $t>0$ satisfying the sufficient decrease follows immediately from \Cref{lem:deltaandf'}. 
The method of proof to follow adapts the step-size arguments given in Royer and Wright \cite{royer2018complexity} to the convex-composite setting. Similar ideas on convex majorants for the composite \ref{theprogram} are employed in \cite{lewis2016proximal,2016arXiv160500125D}.
\begin{algorithm}[H]
	\caption{Global Backtracking}
	\label{alg:btglobal}
	\begin{algorithmic}[1]
		\Procedure{BacktrackingGlobal}{$x^0, \sigma_1, \theta$}
		\State $k\gets 0$\;
		\Repeat
		\State Find $d^k\in\R^n$ such that $\Delta f(x^k; d^k)<0$
		\If{no such $d^k$}
		\State $0\in \sd f(x^k)$\;
		\Return
		\EndIf
		\State $t\gets 1$\;
		\While{$f(x^k + td^k) > f(x^k) + \sigma_1t\Delta f(x^k;d^k)$}
		\State $t\gets \theta t$\;
		\EndWhile
		\State $t_k \gets t$\;
		\State $x^k \gets x^k + t_kd^k$\;
		\State $k\gets k+1$\;
		\Until{}
		\EndProcedure
	\end{algorithmic}
\end{algorithm}
\begin{theorem}\label{thm:btglobal}
	Let $f$ be as in \ref{theprogram}, $x^0\in\dom{g}$, $0<\sigma_1<1$, and $0<\theta<1$. Set $\cL:=\lev_f(f(x^0))$. Suppose there exists $M>0$ and $\tM>0$ such that $\norm{d^k}\leq M$, $\sup_{x\in \cL}\norm{\nabla c(x)}\leq \tM$, and that 
	\begin{enumerate}[label=(\roman*)]
		\item $\nabla c$ is $L_{\nabla c}$-Lipschitz on $\cL+M\bB_n$;
		\item $h$ is $L_h$-Lipschitz on $c(\cL+M\bB) + \tM M\bB_m$.
	\end{enumerate}
	Let $\set{x^k}$ be a sequence initialized at $x^0$ and generated by \Cref{alg:btglobal}:
	Then one of the following must occur:
	\begin{enumerate}[label=(\alph*)]
		\item the algorithm terminates finitely at a first-order stationary point for $f$;
		\item $f(x^k)\searrow-\infty$;
		\item $\displaystyle\sum_{k=0}^\infty\frac{\Delta f(x^k;d^k)^2}{\norm{d^k}_2^2}< \infty$, in particular, $\Delta f(x^k;d^k)\to0$.
		\end{enumerate}
\end{theorem}
\begin{proof}
	We assume (a) - (b) do not occur and show (c) occurs. Since (a) does not occur, the sequence $\set{x^k}$ is infinite, and $\Delta f(x^k; d^k) < 0$ for all $k \ge0 $. The sufficient decrease \eqref{WWI} obtained by the backtracking subroutine gives a strict descent method, so the function values $\set{f(x^k)}$ are strictly decreasing, with $\set{x^k}\subset \cL$ for all $k\ge0$. In particular, $f(x^k)\searrow \bar f > -\infty$.

	We first show that for each $k\geq0$, the step size $0<t_k\leq1$ satisfies 
	\begin{equation}
	\label{eq:BTlowerboundstepsize}
	t_k\geq\min\set{1, \frac{\mu(1-\sigma_2)|\Delta f(x^k;d^k)|}{L_{\nabla c}L_h\norm{d^k}^2}},
	\end{equation}
	by considering two cases.
	
	If the unit step $t_k=1$ is accepted, the bound is immediate. Following \cite{royer2018complexity}, suppose now that the unit step length is not accepted. Then $\hat t:=\theta^j\in(0,1]$ does not satisfy the decrease condition for some $j\ge0$. Using the Lipschitz condition on $h$, the quadratic bound lemma, and \Cref{lem:deltaandf'}, we obtain
	\begin{align*}
		\sigma_1\hat t \Delta f(x^k;d^k) < f(x^k+\hat t d^k) - f(x) &\le \Delta f(x^k;\hat t d^k) + \frac{L_{\nabla c}L_h}{2}\norm{\hat t d^k}_2^2 \\
		& \le \hat t\Delta f(x^k;d^k) + (\hat{t})^2\frac{L_{\nabla c}L_h}{2}\norm{d^k}^2
	\end{align*}
	After dividing both sides by $\hat t>0$ and rearranging,
	\begin{equation}
		\hat t
		\geq \frac{2(1-\sigma_1)|\Delta f(x^k;d^k)|}{L_{\nabla c}L_h\norm{d^k}_2^2}.
	\end{equation}
	Consequently, when the backtracking algorithm terminates at $t_k>0$,
	\begin{equation}\label{eq:tklargeBT}
		t_k\geq \frac{2\theta(1-\sigma_1)|\Delta f(x^k;d^k)|}{L_{\nabla c}L_h\norm{d^k}_2^2}.
	\end{equation}
	Therefore, $t_k$ satisfying \eqref{WWI} implies
	\[
		\sigma_1 \min\set{1,\theta \frac{2(1-\sigma_1)|\Delta f(x;d)|}{L_{\nabla c}L_h\norm{d}_2^2}}|\Delta f(x;d)| \leq \sigma_1t_k|\Delta f(x^k;d^k)| \leq f(x^k)-f(x^{k+1}).
	\]
	Using the boundedness of the search directions and arguing as in the proof of \Cref{thm:wwglobal}, the bound \eqref{eq:tklargeBT} holds for all $k\geq k_0$. Summing the previous display,
	\[
	0 < \sum_{k\geq k_0}\theta \frac{2\sigma_1(1-\sigma_1)\Delta f(x^k;d^k)^2}{L_{\nabla c}L_h\norm{d^k}_2^2} < f(x^0) - \lim_{k\to\infty} f(x^k).
	\]
	Since (b) does not occur, $\lim_{k\to\infty} f(x^k) > -\infty$, so (c) must occur.
\end{proof}
\begin{remark}
	When $h$ is the identity on $\R$ and $g=0$, we recover the convergence analysis of backtracking for smooth minimization.
\end{remark}
The following corollary is an immediate consequence of \Cref{lem:delfto0}.
\begin{corollary}
	Let the hypotheses of \Cref{thm:btglobal} hold. If $0<\beta<1$ and the directions $\set{d^k}$ are chosen to satisfy
	\[
		\Delta f(x^k;d^k)\leq\beta\bar\Delta_k f<0,
	\]
	then the occurrence of (c) in \Cref{thm:btglobal} implies that cluster points of $\set{x^k}$ are first-order stationary for \ref{theprogram}.
\end{corollary}
\section{Weak Wolfe for Convex-Composite Minimization}
\begin{definition}
Weak Wolfe in the convex composite case is defined at $x\in\dom{g}$ with $\Delta f(x;d)<0$ by choosing $0<\sigma_1<\sigma_2<1$ and $\mu > 0$ and requiring
\begin{align}
  \tag{WWI}\label{WWI}
  f(x+td) &\leq f(x) + \sigma_1t\Delta f(x;d), \text{ and} \\
  \tag{WWII}\label{WWII}
  \sigma_2\Delta f(x;d) &\leq \frac{\Delta f(x+td;\mu d)}{\mu}\ .
\end{align}	
\end{definition}
\begin{remark}
The second condition \eqref{WWII} is a \emph{curvature condition} that parallels the classical weak Wolfe \cite{wolfe1969convergence,wolfe1971convergence} curvature condition for smooth, unconstrained minimization:
\[
	\sigma_2 f'(x;d) \leq f'(x + td; d),
\]
which prevents the line search early termination at ``strongly negative" slopes \cite[Section 3.1]{wright1999numerical}.
\end{remark}
\begin{remark}
	The strong Wolfe conditions require $|f'(x+td;d)| \le -\sigma_2f'(x;d)$, whenever $f$ is smooth. However, in nonsmooth minimization, kinks and upward cusps at local minimizers make this condition unworkable.
\end{remark}
The following lemma shows that the set of points satisfying \eqref{WWI} and \eqref{WWII} has nonempty interior.
\begin{lemma}
Let $f$ be as in \ref{theprogram}, $x\in\dom{g}$, and $d$ chosen so that $\Delta f(x;d) < 0$. Suppose $f$ is bounded below on the ray $\{x+td:t>0\}$, and $\mu\in\R$. Then, the set
\[
  C(\mu):=\bset{t > 0}{\begin{aligned}
  f(x+td) &\leq f(x) + \sigma_1t\Delta f(x;d),\\
   \sigma_2\Delta f(x;d) &\leq \frac{\Delta f(x+td;\mu d)}{\mu}
  \end{aligned}}
\]
has nonempty interior for any $\mu>0$.
\end{lemma}
\begin{proof}
Define
\begin{align*}
K(y, z, t) &:= h(y) + g(z) - [f(x) + \sigma_1t\Delta f(x;d)],\\
  G(t) &:= \begin{pmatrix}
    c(x+td) \\
    x + td \\
t
\end{pmatrix}, \text{ with } G'(t) = \begin{pmatrix}
\nabla c(x+td)d \\
d \\
1
\end{pmatrix},
\end{align*}
and set $\phi(t) := K(G(t)) = f(x+td) - [f(x) + \sigma_1t\Delta f(x;d)]$. Then, $\phi(t)$ is convex-composite,
\begin{align*}
	\Delta \phi(t;\mu) &= K(G(t) + G'(t)\mu) - K(G(t)) \\
	&= h(c(x+td) + \nabla c(x+td)\mu d) + g(x+(t+\mu)d) - [f(x)+\sigma_1(t+\mu)\Delta f(x;d)] \\
	&-(h(c(x+td)) + g(x+td) - [f(x)+\sigma_1t\Delta f(x;d)]) \\
	&= \Delta f(x+td; \mu d) - \mu \sigma_1 \Delta f(x;d),
\end{align*}
and, by \Cref{lem:deltaandf'},
\begin{align*}
\phi'(t;\mu) &= f'(x+td; \mu d) - \mu \sigma_1\Delta f(x;d)\\
& \le \mu\Delta\phi(t;1).
\end{align*}
Consequently, $\phi'(0;1)\le (1-\sigma_1)\Delta f(x;d)<0$, so there exists $\bt>0$ such that for all $t\in (0,\bar{t})$,  $\phi(t) < 0$. This is equivalent to \eqref{WWI} being satisfied on $(0,\bt)$.

Since $\phi$ is bounded below on the ray, $\phi(t)\nearrow\infty$. Let $\hat{t}:=\sup\{t > \bar{t}: \phi(s) < 0\text{ for all }s\in(0,t)\}$. Then, since $g$ is closed and $h$ is finite-valued,
\(
	\phi(\hat t)=\liminf_{t\nearrow\hat{t}} \phi(t),
\)
which implies
\begin{align*}
	h(c(x+\hat{t}d)) + g(x+\hat{t}d) &= -[f(x) + \sigma_1\hat{t}\Delta f(x;d)] + \liminf_{t\nearrow\hat{t}} \phi(t)\\
	&\le -[f(x) + \sigma_1\hat{t}\Delta f(x;d)] < \infty,
\end{align*}
so $x+\hat{t}d\in\dom{g}$. Since $g$ is continuous relative to its domain, $\phi$ is continuous relative to its domain, so $\phi(\hat{t})\le0$. We now consider two cases on the value of $\phi(\hat t)$. 

Suppose $\phi(\hat t) < 0$. We aim to show that $f$ satisfies \eqref{WWI} and \eqref{WWII} on the interval $((\hat t - \mu)_+,\hat t]$. To prove this, we show that 
if $\phi(\hat{t})<0$, then $t>\hat{t}$ implies $x+td\not\in\dom{g}$ and, as a consequence, $\hat t \ge1$. Suppose to the contrary that there exists $t>\hat{t}$ with $x+td\in\dom{g}$. Then, the definition of $\hat t$, convexity of $\dom\phi$, and the intermediate value theorem imply there exists $\tilde{t}$ such that $\phi(\tilde{t})=0$ and $t\ge\tilde{t}>\hat{t}$. But relative continuity of $\phi$ at $\tilde{t}$ with respect to $\dom\phi$ means there exist points in $(\hat{t},\tilde{t})$ which contradict the definition of $\hat{t}$. This proves the claim. Consequently, if $\phi(\hat{t})<0$, then $f$ satisfies both \eqref{WWI} and \eqref{WWII} on the interval $((\hat{t}-\mu)_+, \hat{t}]$, as the right-hand side of \eqref{WWII} is $+\infty$.

Otherwise, $\phi(\hat{t})=0$. Let $\tilde{t}\in\argmin_{t\in[0,\hat{t}]} \phi(t)$. Then, $\tilde{t}\in (0,\hat{t})$, with $0\le\phi'(\tilde{t};\mu)\le\Delta \phi(\tilde{t},\mu)$ for all $\mu$, equivalently
\[
	\frac{\Delta f(x+\tilde{t}d; \mu d)}{\mu}\geq \sigma_1\Delta f(x;d) > \sigma_2 \Delta f(x;d)\quad \forall\,\mu > 0,
\]
so \eqref{WWI} and \eqref{WWII} hold with strict inequality at $\tilde{t}$. We now consider two cases based on whether $x+(\tilde{t}+\mu)d\in\dom{g}$.

First, for all sufficiently small $\mu>0$, $x+(\tilde{t}+\mu)d\in\dom{g}$. Because the inequalities in \eqref{WWI} and \eqref{WWII} are strict at $\tilde t$, relative continuity of $f$ and of $t\mapsto\Delta f(x+td;d)$ at $t=\tilde{t}$ imply there exists an open interval $\cI$ with $\tilde{t}\in \cI$ and $x+\cI d\subset\dom{g}$ where both \eqref{WWI} and \eqref{WWII} hold.

For those $\mu>0$ for which $x + (\tilde{t}+\mu)d\not\in\dom{g}$, \eqref{WWI} and \eqref{WWII} hold for all $t\in ((\tilde{t}-\mu)_+, \hat{t})$ as argued in the previous case where $\phi(\hat t)<0$.
\end{proof}
%
Next, we prove finite termination of a bisection algorithm to point $\bar{t}\ge0$ satisfying the weak Wolfe conditions. The algorithm is analogous to the weak Wolfe bisection method for finite-valued nonsmooth minimization in \cite{lewis2013nonsmooth}.
    \begin{algorithm}[H]
    \caption{Weak Wolfe Bisection Method}    
    \label{alg:wwbisect}
    \begin{algorithmic}[1] 
    \Require $x\in\dom{g}, d\in\R^n$ with $\Delta f(x;d)<0$, and $0<\sigma_1<\sigma_2<1,\ \mu>0$.
\Procedure{WWBisect}{$x, d, \sigma_1, \sigma_2$}
        \State $\alpha \gets 0$;
        \State $t\gets 1$;
        \State $\beta \gets \infty$;
        \While{\eqref{WWI} and \eqref{WWII} fail}
        	\If{$f(x+td) > f(x) + \sigma_1 t \Delta f(x;d)$}
        	\Comment{If not sufficient decrease}
        		\State $\beta\gets t$\;
   			\ElsIf{$\sigma_2\Delta f(x;d) > \frac{\Delta f(x+td;\mu d)}{\mu}$}
   			\Comment{Else if not curvature}
			\State $\alpha\gets t$\;
  			\Else
  				\State return $t$\;
  			\EndIf
  			\If{$\beta=\infty$}
  			\Comment{Doubling Phase}
			\State $t\gets 2t$\;
			\Else
			\Comment{Bisection Phase}
				\State $t\gets \frac{1}{2}(\alpha+\beta)$\;
			\EndIf

  		\EndWhile        
  		\EndProcedure
\end{algorithmic}
\end{algorithm}
\begin{lemma} \label{lem:wwbisect}
	Let $f$ be given as in \ref{theprogram} with $g$ strictly continuous relative to its domain, and suppose $x\in\dom{g}$ and $d$ is chosen such that $\Delta f(x;d)<0$. Then, one of the following must occur in \Cref{alg:wwbisect}:
	\begin{enumerate}[label=(\alph*)]
		\item the doubling phase does not terminate finitely, with the parameter $\beta$ never set to a finite value, the parameter $\alpha$ becoming positive on the first iteration and doubling every iteration thereafter, with
$f(x + t_kd)\searrow -\infty$;
		\item both the doubling phase and the bisection phase terminate finitely to a $\bt\ge0$ for which the weak Wolfe conditions are satisfied.
	\end{enumerate}
\end{lemma}

\begin{proof}
	Suppose the procedure does not terminate finitely. If the parameter $\beta$ is never set to a finite value, then the doubling phase does not terminate. Then the parameter $\alpha$ becomes positive on the first iteration and doubles on each subsequent iteration $k$, with $t_k$ satisfying
	\[
		f(x+t_kd) \leq f(x) + \sigma_1t_k\Delta f(x;d),\quad \forall\,k\ge1.
	\]
	Therefore, since $\Delta f(x;d)<0$, the function values $f(x+t_kd)\searrow-\infty$, so the first option occurs.
	
	Otherwise, the procedure does not terminate finitely, and $\beta$ is eventually finite. Therefore, the doubling phase terminates finitely, but the bisection phase does not terminate finitely. This implies 
	there exists $\bar{t} \geq 0$ such that
	\begin{equation}
	\label{wolfebisection}
		\alpha_k \nearrow \bar{t}, \quad t_k\to\bar{t}, \quad \beta_k\searrow \bar{t}.
	\end{equation}		
	We now consider two cases. First, suppose that the parameter $\alpha$ is never set to a positive number. Then, $\alpha_k=0$ for all $k\ge1$, and $t_k, \beta_k\to0$, so the first \texttt{if} statement is entered in each iteration. This implies
	\[
		 \sigma_1\Delta f(x;d) < \frac{f(x+t_kd)-f(x)}{t_k},\quad\forall\,k\ge1.
	\]
	Since $[x,x+d]\subset\dom{g}$, \Cref{lem:deltaandf'}  yields the chain of inequalities
	\[
		\sigma_1 \Delta f(x;d)\leq f'(x;d) \le \Delta f(x;d) < 0,
	\] 
	which contradicts $\sigma_1\in (0,1)$.
		
	Otherwise, the parameter $\alpha$ is eventually positive. Then, the bisection phase does not terminate, and the algorithm generates infinite
	sequences $\set{\alpha_k}, \set{t_k},$ and $\set{\beta_k}$ satisfying \eqref{wolfebisection} such that, for all $k$ large, $0<\alpha_k<t_k<\beta_k<\infty$, and
	\begin{align}
		f(x+\alpha_k d) &\leq f(x) + \sigma_1\alpha_k \Delta f(x;d), \label{eq:sufficientalpha}\\
		f(x+\beta_k d) &> f(x) + \sigma_1\beta_k \Delta f(x;d),\label{eq:notsufficientbeta} \\
		\sigma_2\Delta f(x;d) &> \frac{\Delta f(x + \alpha_k d ; \mu d)}{\mu}, \label{eq:notcurvaturealpha}\\
		[x,x+\max\set{\alpha_k+\mu,\beta_k}d] &\subset \dom{g}. \label{eq:alphadom}
	\end{align}
Letting $k\to\infty$ in \eqref{eq:notcurvaturealpha} and using lower semicontinuity of $g$ gives
	\begin{equation}
	\label{eq:curvaturecontradiction}
		\sigma_2 \Delta f(x;d) \geq \frac{\Delta f(x+\bar{t}d; \mu d)}{\mu}.
	\end{equation}
	By \Cref{lem:mvt},
 for sufficiently large $k$ 
there exists $\tau_k\in(0,1)$ so that the vectors 
\begin{align*}
	x^k &:= (1-\tau_k)(x+\alpha_k d) + \tau_k(x+\beta_k d) = x + [(1-\tau_k)\alpha_k + \tau_k\beta_k]d, \\
	v^k &\in \sd f(x^k)
\end{align*}
yield an extended form of the mean-value theorem
	\begin{equation}\label{eq:generalmvt}
		f(x+\beta_k d) - f(x + \alpha_k d) = \ip{v^k}{(\beta_k-\alpha_k)d}. 
	\end{equation}
	Let $\gamma_k := (1-\tau_k)\alpha_k + \tau_k\beta_k\in (\alpha_k,\beta_k)$, so that $x^k=x+\gamma_kd$. Then, $\gamma_k\to\bar{t}$ as $k\to\infty$. Combining \eqref{eq:sufficientalpha} and \eqref{eq:notsufficientbeta} and using \eqref{eq:generalmvt} gives
	\begin{align*}
		\sigma_1(\beta_k-\alpha_k)\Delta f(x;d) &< f(x+\beta_k d) - f(x + \alpha_k d) = \ip{v^k}{(\beta_k-\alpha_k)d}.
	\end{align*}
	Dividing by $\beta_k-\alpha_k>0$ gives 
	\begin{align*}
		\sigma_1\Delta f(x;d) &\le f'(x+\gamma_kd;d)\\ 
		&\le \frac{\Delta f(x+\gamma_kd;\mu d)}{\mu}. 
	\end{align*}
	As $k\to\infty$, using \eqref{eq:curvaturecontradiction}, we obtain the string of inequalities
	\[
		 \frac{\Delta f(x+\bar{t}d; \mu d)}{\mu} \leq \sigma_2\Delta f(x;d)<\sigma_1\Delta f(x;d) \leq \frac{\Delta f(x+\bar{t}d;\mu d)}{\mu},
	\]
	which is a contradiction. Therefore, either the doubling phase never terminates or the procedure terminates finitely at some $\bt$ at which $f$ satisfies both weak Wolfe conditions.
\end{proof}
A global convergence result for the weak Wolfe line search that parallels \cite[Theorem 2.4]{burke1985descent} now follows under standard Lipschitz assumptions, which hold, in particular, if the initial set $\lev_f(f(x^0))$ is compact.
	\begin{algorithm}[H]
	\caption{Global Weak Wolfe}
	\label{alg:wwglobal}
	\begin{algorithmic}[1]
		\Procedure{WeakWolfeGlobal}{$x^0, \sigma_1, \sigma_2, \mu$}
		\State $k\gets 0$\;
		\Repeat
		\State Find $d^k\in\R^n$ such that $\Delta f(x^k; d^k)<0$
		\If{no such $d^k$}
		\State $0\in \sd f(x^k)$\;
		\Return
		\EndIf
		\State Let $t_k$ be a step size satisfying \eqref{WWI} and \eqref{WWII}
		\If{no such $t_k$}
		\State $f$ unbounded below.\;
		\Return
		\EndIf
		\State $x^k \gets x^k + t_kd^k$\;
		\State $k\gets k+1$\;
		\Until{}
		\EndProcedure
	\end{algorithmic}
\end{algorithm}
\begin{theorem}\label{thm:wwglobal}	Let $f$ be as in \ref{theprogram} with $g$ strictly continuous relative to its domain, $x^0\in\dom{g}$, $0<\sigma_1<\sigma_2<1$, and $0<\mu<1$. Set $\cL:=\lev_f(f(x^0))$. Suppose there exists $M,\tM>0$ such that $\norm{d^k}\leq M$ for all $k\ge0$, $\sup_{x\in\cL}\norm{\nabla c(x)}\leq \tM$, and
	\begin{enumerate}[label=(\roman*)]
		\item $c$ is $L_c$-Lipschitz on $\cL$;
		\item $\nabla c$ is $L_{\nabla c}$-Lipschitz on $\cL$;
		\item $g$ is $L_g$-Lipschitz on $(\cL + M\mu\bB)\cap\dom{g}$;
		\item $h$ is $L_h$-Lipschitz on $c(\cL)+M\tM\mu\bB$.
	\end{enumerate}
	Let $\set{x^k}$ be a sequence initialized at $x^0$ and generated by \Cref{alg:wwglobal}:
Then at least one of the following must occur:
\begin{enumerate}[label=(\alph*)]
	\item the algorithm terminates finitely at a first-order stationary point for $f$;
	\item for some $k$ the step size selection procedure generates a sequence of trial step sizes 
$t_{k_n}\overset{n\uparrow\infty}{\longrightarrow}\infty$
such that $f(x^k + t_{k_n}d^k)\to-\infty$;
	\item $f(x^k)\searrow-\infty$;
	\item $\displaystyle\sum_{k=0}^\infty\frac{\Delta f(x^k;d^k)^2}{\norm{d^k}+\norm{d^k}^2}< \infty$, in particular, $\Delta f(x^k;d^k)\to0$.
\end{enumerate}
\end{theorem}
\begin{proof}
We assume (a) - (c) do not occur and show (d) occurs. Since (a) does not occur, the sequence $\set{x^k}$ is infinite, and $\Delta f(x^k; d^k) < 0$ for all $k \ge0 $. Since (b) does not occur, \Cref{lem:wwbisect} implies that the weak Wolfe bisection method terminates finitely at every iteration $k\ge0$. The sufficient decrease condition \eqref{WWI} gives a strict descent method, so the function values $\set{f(x^k)}$ are strictly decreasing, with $\set{x^k}\subset \cL$ for all $k\ge0$. By the nonoccurrence of (c), $f(x^k)\searrow \bar f > -\infty$.

We first show that for each $k\geq0$, the step size $t_k$ satisfies 
	\begin{equation}
	\label{eq:lowerboundstepsize}
		t_k\geq\min\set{1-\mu, \frac{\mu(1-\sigma_2)|\Delta f(x^k;d^k)|}{K\left(\norm{d^k}+\norm{d^k}^2\right)}},
	\end{equation}
by considering two cases.

First, suppose $\Delta  f(x^{k+1};\mu d^k)=\infty$. Then $x^{k+1}+\mu d^k=x^k+(t_k+\mu)d^k\not\in\dom{g}$. Since $x^k+d^k\in \dom{g},\ t_k+\mu>1$, and by assumption $0<\mu<1$. Therefore, $t_k\ge 1-\mu$.

Otherwise, $\Delta f(x^{k+1};\mu d^k)<\infty$. Then
	\begin{align*}
\Delta f(x^{k+1};\mu d^k) - \Delta f(x^k;\mu d^k)
&=h(c(x^{k+1}) + \nabla c(x^{k+1})\mu d^k) - h(c(x^{k+1})) + g(x^{k+1} + \mu d^k) - g(x^{k+1}) \\
&- [h(c(x^k) + \nabla c(x^k)\mu d^k) - h(c(x^k)) + g(x^k + \mu d^k) - g(x^k)] \\
&= h(c(x^k)) - h(c(x^{k+1})) \\
& + h(c(x^{k+1}) + \nabla c(x^{k+1})\mu d^k) - h(c(x^k) + \nabla c(x^k)\mu d^k) \\
& + g(x^k) - g(x^{k+1}) + g(x^{k+1} + \mu d^k) - g(x^k + \mu d^k) \\
&\le 2L_hL_ct_k\norm{d^k} + L_hL_{\nabla c}\mu t_k\norm{d^k}^2 + 2L_gt_k\norm{d^k}\\
&\le Kt_k\left(\norm{d^k}+\norm{d^k}^2\right),
\end{align*}
for some $K\ge0$. Adding and subtracting in \eqref{WWII} gives
	\begin{align*}
\sigma_2\Delta f(x^k;d^k) &\leq \frac{\Delta f(x^k+t_kd^k;\mu d^k)}{\mu} \\
&= \frac{\Delta f(x^k;\mu d^k)}{\mu} + \left[\frac{\Delta f(x^k+t_kd^k;\mu d^k)}{\mu} - \frac{\Delta f(x^k;\mu d^k)}{\mu}\right]\\
&\leq \Delta f(x^k;d^k) + \frac{K}{\mu}t_k\left(\norm{d^k}+\norm{d^k}^2\right) (\text{since }0<\mu<1),
\end{align*}
which rearranges to
\begin{equation}\label{eq:tklarge}
0 < \frac{\mu(1-\sigma_2)|\Delta f(x^k;d^k)|}{K\left(\norm{d^k}+\norm{d^k}^2\right)} \leq t_k,
\end{equation}
so \eqref{eq:lowerboundstepsize} holds. Next, \eqref{WWI} and \eqref{eq:lowerboundstepsize} imply 
	\begin{equation}
	\label{eq:suffdecrstepsize}
	\small{\sigma_1\min\set{1-\mu, \frac{\mu(1-\sigma_2)|\Delta f(x^k;d^k)|}{K\left(\norm{d^k}+\norm{d^k}^2\right)}}|\Delta f(x^k;d^k)| \leq \sigma_1t_k |\Delta f(x^k;d^k)| \leq f(x^k)-f(x^{k+1}).}
	\end{equation} 
We aim to show that the bound \eqref{eq:tklarge} holds for all large $k$ by showing $\Delta f(x^k;d^k)\to0$ and using boundedness of the search directions $\set{d^k}$. Suppose there exists a subsequence
$J_1\subset \bN$ for which $\Delta f(x^k;d^k)\not\xrightarrow[J_1]{} 0$. Let $\gamma>0$ be such that $\sup_{k\in J_1} \Delta f(x^k;d^k) \leq -\gamma < 0$. Then, since $\set{d^k}\subset M\bB$,
\begin{equation}
\label{eq:rationot}
	\frac{\mu(1-\sigma_2)|\Delta f(x^k;d^k)|}{K\left(\norm{d^k}+\norm{d^k}^2\right)} \not\xrightarrow[J_1]{} 0.
\end{equation}
If there exists a further subsequence $J_2\subset J_1$ with
\[
	\frac{\mu(1-\sigma_2)|\Delta f(x^{k};d^{k})|}{K\left(\norm{d^{k}}+\norm{d^{k}}^2\right)} \ge 1-\mu,\quad \forall\,k\in J_2,
\] 
then by expanding the recurrence given by \eqref{WWI}, and writing $J_2 = \set{k_1,k_2,\dotsc}$, we have 
\begin{align}
	\label{eq:recurrbounds}f(x^{k_n}) &\leq f(x^{k_{n-1}}) - \sigma_1(1-\mu)\gamma \\
	\nonumber &\leq f(x^{k_1}) - C(k_n)\sigma_1(1-\mu)\gamma
\end{align}
with $C(k_n)\to\infty$ as $n\to\infty$. This contradicts the nonoccurence of (c). By \eqref{eq:rationot}, there exists a subsequence $J_2\subset J_1$ and $\delta>0$ so that
\[
	0<\delta\le\frac{\mu(1-\sigma_2)|\Delta f(x^{k};d^{k})|}{K\left(\norm{d^{k}}+\norm{d^{k}}^2\right)} < 1-\mu
\]
for all large $k\in J_2$. Repeating the argument at \eqref{eq:recurrbounds} with $\delta$ in place of $1-\mu$, we conclude 
\[
	\frac{\mu(1-\sigma_2)|\Delta f(x^k;d^k)|}{K\left(\norm{d^k}+\norm{d^k}^2\right)} \xrightarrow[J_1]{} 0,
\]
and consequently $\Delta f(x^k;d^k)\xrightarrow[J_1]{}0$, which is a contradiction. Therefore, \eqref{eq:tklarge} holds for all $k\ge k_0$. Summing over $k\in \bN$ in \eqref{eq:suffdecrstepsize} gives
	\[
		0 < \sum_{k\ge k_0}\frac{\sigma_1\mu(1-\sigma_2)\Delta f(x^k;d^k)^2}{K\left(\norm{d^k}+\norm{d^k}^2\right)} < f(x^0) - \lim_{k\to\infty} f(x^k).
	\]
Since (c) does not occur, $\lim_{k\to\infty} f(x^k) > -\infty$, so (d) must occur.
\end{proof}
\begin{remark}
	When $h$ is the identity on $\R$ and $g=0$, we recover the convergence analysis of weak Wolfe for smooth minimization given in \cite[Theorem 3.2]{wright1999numerical}.
\end{remark}
\begin{remark}
	The hypotheses of \Cref{thm:wwglobal} simplify if $h$ is globally Lipschitz. In that case, the boundedness condition on $\bset{\norm{\nabla c(x)}}{x\in\cL}$ is not necessary. Alternatively, if $\norm{\nabla c(x)}$ is bounded on the closed convex hull of $\cL$, then the Lipschitz condition of $c$ on $\cL$ is immediate.
\end{remark}
The following corollary is an immediate consequence of \Cref{lem:delfto0}.
\begin{corollary}
	Let the hypotheses of \Cref{thm:wwglobal} hold. If $0<\beta<1$ and the directions $\set{d^k}$ are chosen to satisfy
	\[
		\Delta f(x^k;d^k)\leq\beta\bar\Delta_k f<0,
	\]
	then the occurrence of (d) in \Cref{thm:wwglobal} implies that cluster points of $\set{x^k}$ are first-order stationary for \ref{theprogram}.
\end{corollary}
\section{Trust-Region Subproblems}
In this section, we let $\norm{\cdot}$ denote an arbitrary norm on $\R^n$.
\begin{definition}
For $\delta>0$ and $x\in\dom{g}$, define the set of Cauchy steps $\Dcd{\delta}(x)$ by  
\[
	\Dcd{\delta}(x):=\argmin_{\norm{d}\leq \delta} \Delta f(x;d),
\]
and set 
\[
\Deltac{\delta} f(x):=\inf_{\norm{d}\leq \delta} \Delta f(x;d).
\]
\end{definition}
Observe that $\Delta^{\scriptscriptstyle{C}}_{\eta_k} f(x^k) = \overline\Delta_k f$, where $\overline \Delta_k f$ is defined in \Cref{lem:delfto0}. Our pattern of proof in this section follows a path similar to the standard approaches to such results given in \cite{conn2000trust}. In short, the Cauchy step defined above provides the driver for the global convergence of \Cref{alg:trs}.
The particular steps chosen by the algorithm must do at
least as well as the Cauchy step. This idea is formalized by the
notion of the \emph{sufficient decrease condition} described below.
The utility of this condition requires some basic Lipschitz continuity assumptions.

\noindent
{\bf Assumption:}
\begin{itemize}
\item[A1:]
 $c$, $\nabla c$, and $g$ are Lipschitz continuous on $\dom{g}$ with
 Lipschitz constants $L_c,\ L_{c' },$ and $L_g$, respectively.
 \item[A2:]
 $h$ is Lipschitz continuous on $g(\dom{g})$ with Lipschitz constant $L_h$.
 \end{itemize}
 
 These assumptions imply the Lipschitz continuity of $\Deltac{\delta}f$
 on $\dom{g}$.
 
 \begin{lemma}\label{lem:del f Lip}
 Let the assumptions A1 and A2 hold. Then, for $\delta >0$, the mapping
 $x\mapsto \Deltac{\delta}f(x)$ is Lipschitz continuous on $\dom{g}$ with
 constant $L_\Del:=L_h(2L_c+\delta L_{c'})+2L_g$.
 \end{lemma}
 \begin{proof}
 Observe that for any $x^1,x^2\in\dom{g}$ and $d\in\delta\bB$
 we have
 \[
 \begin{aligned}
 \Del f(x^1;d)-\Del f(x^2;d)=&
 [h(c(x_1)+\nabla c(x_1)d)-h(c(x_2)+\nabla c(x_2)d)]+[h(c(x_2))-h(c(x_1))]\\
 &\quad [g(x^1+d)-g(x^2+d)]+[g(x^2)-g(x^1)].
 \end{aligned}
 \]
 Hence, 
 \[
 \Del f(x^1;d)-\Del f(x^2;d)\le L_\Del\norm{x^1-x^2},
 \]
 and, by symmetry,
 \[
  \Del f(x^2;d)-\Del f(x^1;d)\le L_\Del\norm{x^1-x^2}.
 \]
 Taking $d\in \Dcd{\delta}(x^2)$ in the first of these inequalities and 
 $d\in \Dcd{\delta}(x^1)$ in the second gives the result.
 \end{proof}
 
\begin{definition}
(Sufficient Decrease Condition for $f$)
\label{def:trsufficientdecrease}
We say that a direction choice method satisfies the sufficient decrease condition for 
$f$ if
for all $\epsilon>0$ and $\delta_k>0$, if $x^k\in\Rn$ satisfies $|\Deltac{1} f(x^k)| > \eps$, then
there exists constants $\kappa_1,\kappa_2>0$ depending only on $\eps$ such that
the direction choice $d^k\in\Rn$ satisfies
	\begin{equation}\label{eq:btr}
		\Delta f(x^k;d^k) + \frac{1}{2}d^{k\top}H_kd^k < -\kappa_1\min(\kappa_2,\delta_k).
	\end{equation}
\end{definition}

\begin{remark}
The sufficient decrease condition can be defined using $|\Deltac{\hdel} f(x^k)|$
for any $\hdel>0$, but this choice of $\hdel$ must then remain constant
throughout the iteration process; $\hdel=1$ is chosen for simplicity.
\end{remark}

We now show that the sufficient decrease condition can always be 
satisfied in a neighborhood of any non-stationary point.

\begin{lemma}\label{lem:suff decrease}
Let $x\in\dom{g}$, $H\in\R^{n\times n}$, $\delta>0$ and let $\sigma>0$ be such that 
$\norm{d}_2\leq \sigma\norm{d}$ for all $d\in\Rn$. 
If $\hat d\in\Dco (x)$ with $\Deltaco f(x)<0$, 
then there exists $\hat t\in(0,\min(1,\delta)]$ such that
	\[
		\Delta f(x;\hat t\hat d) + \frac{\hat t^2}{2}\hat d^\top H\hat d\leq \frac{1}{2}\Deltac{1} f(x) \min\left(\frac{|\Deltac{1}f(x)|}{\sigma^2\norm{H}_2}, 1,\delta \right). 
	\]
\end{lemma}
\begin{proof}
	For any $t\in (0, \min(1,\delta)]$, \Cref{lem:deltaandf'} and H\"older's inequality implies
	\[
		\Delta f(x;t\hat d) + \frac{t^2}{2}\hat d^\top H \hat d \leq t \Delta f(x;\hat d) + \frac{t^2}{2}\sigma^2\norm{H}_2.
	\]
	Set $\alpha = \Delta f(x;\hat d)<0,\ \beta = \sigma^2\norm{H}_2>0$, and 
\[
\hat t=\argmin_{t\in[0,\min(1,\delta)]}\alpha t + \beta t^2/2 
\qquad= \min(1,\delta, -\alpha/\beta).
\]
	
	There are two cases to consider. If $-\alpha/\beta\leq\min(1,\delta)$, then $\hat t = -\alpha/\beta$ and
	\begin{align*}
		\Delta f(x;\hat t\hat d) + \frac{\hat t^2}{2}\hat d^\top H \hat d &\leq \hat t \Delta f(x;\hat d) + \frac{t^2}{2}\hat d^\top H \hat d \\
		&= -\frac{\Delta f(x;\hat d)^2}{\sigma^2\norm{H}_2} + \frac{\Delta f(x;\hat d)^2}{2\sigma^4\norm{H}_2^2}\hat d^\top H \hat d \\
		&= -\frac{\Delta f(x;\hat d)^2}{2\sigma^2\norm{H}_2} \\
		&= \frac{1}{2}\Deltac{1} f(x)\left(\frac{|\Deltac{1} f(x)|}{\sigma^2\norm{H}_2} \right).
	\end{align*}
	Otherwise, $\min(1,\delta)\leq -\alpha/\beta$. Setting $\hat t = \min(1,\delta)$ gives
	\begin{align*}
		\Delta f(x;\hat t\hat d) + \frac{\hat t^2}{2}\hat d^\top H \hat d &\leq \hat t \Delta f(x;\hat d) + \frac{\hat t^2}{2}\sigma^2\norm{H}_2 \\
		&= \hat t \Delta f(x;\hat d) - \frac{\hat t}{2}\Delta f(x;\hat d) \\
		&= \frac{1}{2}\Deltac{1}f(x)\min(1,\delta).
	\end{align*}
	The result follows.
\end{proof}

The sufficient decrease condition 
provides a useful bound on the change in the objective function.

\begin{lemma}\label{lem:ratio achieved}
Let the assumptions A1 and A2 hold.
Suppose 
$\sigma>0$ satisfies $\norm{d}_2\leq \sigma\norm{d}$ for all $d\in\Rn$. Let
$H\in\R^{n\times n}, 0<\overline\beta_1\leq \overline\beta_2 < 1$, and $\alpha,\kappa_1,\kappa_2>0$ be given. Choose $\overline\delta>0$ so that for all $\delta\in[0,\overline\delta]$,
	\[
		\kappa_1(1-\overline\beta_2)\min(\kappa_2,\delta)\ge\frac{1}{2}\delta^2L_hL_{\nabla c} + \frac{1}{2}\sigma^2\delta^2\norm{H}_2.
	\]
	Then, for every $\delta\in[0,\bar\delta]$, $x\in\dom{g}$, and $d\in\delta\bB$ for which
	\[
		\Delta f(x;d) + \frac{1}{2}d^\top H d \leq -\kappa_1\min(\kappa_2,\delta),
	\]
	one has
	\[
		f(x+d)-f(x)\leq \overline\beta_2[\Delta f(x;d)+ \frac{1}{2}d^\top H d]
		\leq \overline\beta_1[\Delta f(x;d)+ \frac{1}{2}d^\top H d].
	\]
\end{lemma}
\begin{proof}
	The Lipschitz assumptions on $h$ and $\nabla c$ imply
	\[
		f(x+d) - f(x) \leq \Delta f(x;d) + \frac{L_hL_{\nabla c}}{2}\norm{d}_2^2.
	\]
	Since $\frac{1}{2}d^\top H d + \frac{1}{2}\sigma^2\delta^2\norm{H}_2\geq0$, it follows that
	\begin{align*}
		f(x+d)-f(x) & \leq \Delta f(x;d) + \frac{1}{2}\delta^2L_hL_{\nabla c}+ \frac{1}{2}d^\top H d + \frac{1}{2}\sigma^2\delta^2\norm{H}_2 \\
		&\leq \Delta f(x;d) + \frac{1}{2}d^\top H d + \kappa_1(1-\overline\beta_2)\min(\kappa_2,\delta)\\
		&\leq \Delta f(x;d) + \frac{1}{2}d^\top H d -(1-\overline\beta_2)[\Delta f(x;d) + \frac{1}{2}d^\top H d ]\\
		&= \overline\beta_2 [\Delta f(x;d) + \frac{1}{2}d^\top H d ] \\
		&<\overline\beta_1 [\Delta f(x;d) + \frac{1}{2}d^\top H d ].
	\end{align*}
\end{proof}
\begin{algorithm}[H]
	\caption{Global Trust Region}
	\label{alg:trs}
	\begin{algorithmic}[1]
	    \Require $x^0\in\dom{g}, H^0\in\R^{n\times n},0<\gamma_1\leq \gamma_2<1\leq\gamma_3,0<\beta_1\leq\beta_2<\beta_3<1, \delta_0>0,$
		\Procedure{TRS}{$x^0,H^0,\gamma_1,\gamma_2,\gamma_3,\beta_1,\beta_2,\beta_3,\delta_0$}
		\State $k\gets 0$\;
		\Repeat
		\State Find $d^k\in
		D_k:=\bset{d}{\norm{d}\leq\delta_k,\ \Delta f(x^k;d) + \frac{1}{2}d^\top H_kd < 0}$
		\If{no such $d^k$}
		\Return
		\EndIf
		\State $r_k\gets\frac{f(x^k+d^k)-f(x^k)}{\Delta f(x^k;d^k) + \frac{1}{2}d^{k\top} H_kd^k}$\;
		\If{$r_k>\beta_3$}
		\State Choose $\delta_{k+1}\in[\delta_k,\gamma_3\delta_k]$
		\ElsIf{$\beta_2\leq r_k\le\beta_3$}
			\State Set $\delta_{k+1}=\delta_k$
		\Else
			\State Choose $\delta_{k+1}\in [\gamma_1\delta_k,\gamma_2\delta_k]$
		\EndIf
		\If{$r_k<\beta_1$}
			\State $x^{k+1}\gets x^k$\;
			\State $H_{k+1} \gets H_k$\;
		\Else
			\State $x^{k+1}\gets x^k+d^k$\;
			\State Choose $H_{k+1}$\;
		\EndIf
		\State $k\gets k+1$\;
		\Until{}
		\EndProcedure
	\end{algorithmic}
\end{algorithm}
The trust-region method of \Cref{alg:trs} selects steps $d^k$ so that 
\[
h(c(x^k)+\nabla c(x^k)d^k)+g(x^k+d^k)+\half (d^k)^\top H_kd^k< f(x^k).
\]
The ratio of actual to predicted reduction in $f$, $r_k$, is then computed, and the trust-region radius is increased or decreased
depending on the value of this ratio. In addition,
if the ratio is sufficiently positive, then the step $d^k$ is
accepted; otherwise it is rejected, and a new step is 
computed using
a smaller trust-region radius. We conclude with a global convergence result based on this method.
\begin{theorem}\label{thm:TR}
Let $f$ be as in \ref{theprogram}, $x^0\in\dom{g}$, and let the hypotheses of Lemma \ref{lem:ratio achieved} be satisfied.
Suppose the sequence $\set{H_k}$ is bounded, and the choice of search directions $\set{d^k}$ generated by \Cref{alg:trs}
satisfy the sufficient decrease condition in \Cref{def:trsufficientdecrease}. Then, one of the following must occur:
	\begin{enumerate}
		\item[(i)] $D_k=\emptyset$ for some $k$;
		\item[(ii)] $f(x^k)\searrow-\infty$;
		\item[(iii)] $|\Deltaco f(x^k)|\to0$.
	\end{enumerate}
\end{theorem}
\begin{proof}
We assume that none of (i) - (iii) occur and derive a contradiction. 
First, observe that Lemma \ref{lem:suff decrease} tells us that if $D_k\ne\emptyset$,
then there is an element of $D_k$ that satisfies the sufficient decrease condition,
and so the algorithm is well-defined and the hypotheses of the theorem can be satisfied.
Consequently, since (i) does not occur, the sequence $\set{x^k}$ is infinite. Since (iii) does not occur, there exists $\zeta>0$ 
such that the set 
\[
J:=\bset{k\in\bN}{2\zeta < |\Deltac{1} f(x^k)|}
\]
is an infinite subsequence of $\bN$. 
By the sufficient decrease condition, there exists $\kappa_1,\kappa_2>0$ such that
\begin{equation}\label{eq:tr1}
		\Delta f(x^k;d^k) + \frac{1}{2}(d^k)^\top H_kd^k \leq -\kappa_1\min(\kappa_2,\delta_k)
		\quad\mbox{whenever }\  |\Deltac{1} f(x^k)| > \zeta.
\end{equation}
In particular, 
	\eqref{eq:tr1} holds for all $k\in J$. 
	Lemma \ref{lem:ratio achieved}
	and Lipschitz continuity of $\nabla c$ imply the existence of a $\bar\delta$ such that
\begin{equation}\label{eq:tr2}
		r_k\geq\beta_2\text{ and }x^{k+1} = x^k + d^k\quad
		\mbox{whenever }\  |\Deltac{1} f(x^k)|>\zeta
		\ \mbox{ and }\ \delta_k\leq  \bar\delta.
\end{equation}
By Lemma \ref{lem:del f Lip}, 
\begin{equation}\label{eq:tr3}
|\Deltaco f(z^1)-\Deltaco f(z^2)|\le \zeta\quad\mbox{whenever }\ 
\norm{z^1-z^2}\le \beps\ \mbox{ with }\ z^1,z^2\in\dom{g},
\end{equation}
where $\beps:=\zeta/L_{\Del}$.

For each $k\in J$, let $\nu(k)$ be the first integer $p\ge k$
such that either
\begin{equation}\label{eq:tr4}
\norm{x^{p+1}-x^k}\le \beps,
\end{equation}
or
\begin{equation}\label{eq:tr5}
\delta_{p}\le \bdel
\end{equation}
is violated.
We need to show that $\nu(k)$ is well-defined for every $k\in J$. 
Assume that $\nu(k)$ is not
well-defined for some $k_0\in J$. That is,  
\begin{equation}\label{eq:tr6}
\norm{x^{p+1}-x^{k_0}}\le \beps\ \mbox{ and }\
\delta_{p}\le \bdel\quad\forall\ p\ge k_0.
\end{equation}
For every $k\in J$ with $k\ge k_0$, observe that \eqref{eq:tr3} and
the first condition
in \eqref{eq:tr6} imply 
that $|\Deltaco f(x^{p})|\ge \zeta$ for all $p>k_0$.
Combining this with 
\eqref{eq:tr1}, \eqref{eq:tr2}, the second condition in \eqref{eq:tr6}
and the definition of $r_k$
gives
\begin{equation}\label{eq:tr7}
\begin{aligned}
f(x^{p+1})-f(x^{p})&\le
\beta_2[\Del f(x^{p};d^{p})
+\half (d^{p})^\top H_{p}d^{p}]\\
&\le -\beta_2\kappa_1\min(\kappa_2,\delta_{p})\\
&=-\beta_2\kappa_1\min(\kappa_2,\delta_{k_0})
\end{aligned}
\end{equation}
for all $p\ge k_0$. But then $f(x^{k})\searrow -\infty$, which contradicts
our working hypotheses. Therefore, 
$\nu(k)$ is well-defined for every $k\in J$.

Let $k\in J$ and suppose $\nu(k)$ is such that \eqref{eq:tr4} is 
violated 
but \eqref{eq:tr5} is not violated 
at $x^{\nu(k)+1}$. 
Then either $\nu(k)=k$ or, as in \eqref{eq:tr7},
\begin{equation}\label{eq:tr8}
\begin{aligned}
f(x^{p+1})-f(x^{p})&\le
\beta_2[\Del f(x^{p};d^{p})
+\half (d^{p})^TH_{p})d^{p}]\\
&\le -\beta_2\kappa_1\min(\kappa_2,\delta_{p}),
\end{aligned}
\end{equation}
for $p=k,\dots,\nu(k)-1$.
If $\nu(k)=k$, then, by \eqref{eq:tr2}, $r_k\ge \beta_2$ and so
\[
f(x^{\nu(k)+1})-f(x^k)\le -\beta_2\kappa_1\min(\kappa_2,\delta_k)
\le -\beta_2\kappa_1\min(\kappa_2,\beps);
\]
otherwise, summing \eqref{eq:tr8} over $p$ gives
\[
\begin{aligned}
f(x^{\nu(k)+1})- f(x^{k})&\le (f(x^{\nu(k)+1})-f(x^{\nu(k)}))-
\sum_{p=k}^{\nu(k)-1}\beta_1\kappa_1\min(\kappa_2,\delta_{p})\\
&\le -\beta_1\kappa_1\min(\kappa_2,\sum_{p=k}^{\nu(k)-1}\delta_{p})\\
&\le -\beta_1\kappa_1\min(\kappa_2,\beps),
\end{aligned}
\]
since $\sum_{p=k}^{\nu(k)-1}\delta_{p}\ge \norm{x^{\nu(k)+1}-x^{k}}
\ge\beps$ (here, the second inequality uses the elementary fact that
for any three non-negative numbers $\tau_1,\tau_2,\tau_3$, we have
$\min(\tau_1,\, \tau_2+\tau_3)\le \min(\tau_1,\, \tau_2)+\min(\tau_1,\, \tau_3)$).
Hence,
\begin{equation}\label{eq:tr9}
f(x^{\nu(k)+1})-f(x^k) 
\le -\beta_2\kappa_1\min(\kappa_2,\beps)
\end{equation}
when \eqref{eq:tr4} is 
violated 
but \eqref{eq:tr5} is not violated. 

Next suppose that \eqref{eq:tr5} is violated at $\nu(k)$, and let $s(k)$ be the 
smallest positive integer for which $x^{\nu(k)+s(k)}\ne x^{\nu(k)}$. 
The integer $s(k)$ is well defined since, by \eqref{eq:tr2}, 
$x^{\nu(k)}$ is eventually updated. Also note that
$x^{(\nu(k)+i)}$ satisfies \eqref{eq:tr2} for $ i=0,1,\dots,s(k)-1$.
Therefore, by \eqref{eq:tr1} and \eqref{eq:tr2},
\begin{equation}\label{eq:tr10}
\begin{aligned}
f(x^{\nu(k)+s(k)})-f(x^{\nu(k)})&\le \beta_2[\Delta f(x^{\nu(k)};d^{\nu(k)+s(k)-1}) 
+ \frac{1}{2}d^{(\nu(k)+s(k)-1)\top}H_{\nu(k)}d^{(\nu(k)+s(k)-1)}]\\
&\le -\beta_2\kappa_1\min(\kappa_2,\, \delta_{\nu(k)+s(k)-1})\\
&\le -\beta_2\kappa_1\min(\kappa_2,\, \gamma_1\bdel),
\end{aligned}
\end{equation}
since $x^{(\nu(k)+i)}=x^{\nu(k)}$ so that 
$\delta_{(\nu(k)+i)+1}=\gamma_3\delta_{(\nu(k)+i)},\, i=0,1,\dots,s(k)-1$, and,
in particular, 
$\delta_{\nu(k)+s(k)-1}=\gamma_3^{-1}\delta_{\nu(k)+s(k)}$.
Set $s(k)=1$ if \eqref{eq:tr4} is 
violated 
but \eqref{eq:tr5} is not violated 
at $x^{\nu(k)+1}$. Putting \eqref{eq:tr9} together with \eqref{eq:tr10}, gives
\[
f(x^{\nu(k)+s(k)})-f(x^{k})\le-\beta_2\kappa_1\min(\kappa_2,\, \gamma_1\bdel,\beps)
\quad\forall\, k\in J.
\]
But then $f(x^k)\searrow -\infty$, which contradicts our working hypotheses.
This establishes the theorem. 
\end{proof}

Under the hypotheses of \Cref{thm:TR}, 
\Cref{lem:del f Lip} tells us that $\Deltaco f$ is Lipschitz
continuous on $\dom{g}$. Hence if $\bar x$ is a cluster point
of the sequence $\{x^k\}$ generated by \Cref{alg:trs},
then $\bar x$ is a stationary point of $f$. We record this fact in the
following corollary.

\begin{corollary}
Under the hypotheses of \Cref{thm:TR}, every cluster
point $\bar x$ of the sequence $\{x^k\}$ generated by \Cref{alg:trs}
has $\Deltaco f(\bar x)=0$, and so is a first-order stationary 
point for $f$.
\end{corollary}

\bibliographystyle{abbrv}
\bibliography{cvxcompbib}
\end{document}